\documentclass[12pt,a4paper]{amsart}
\usepackage{graphics}
\usepackage{epsfig}
\usepackage{graphicx}
\usepackage[all]{xy}
\theoremstyle{plain}
\usepackage{amssymb}
\usepackage{enumerate}

\advance\hoffset-20mm \advance\textwidth41mm

\newtheorem{theorem}{Theorem}
\newtheorem{lemma}{Lemma}
\newtheorem*{theo*}{Theorem}
\newtheorem{proposition}{Proposition}
\newtheorem{corollary}{Corollary}
\theoremstyle{definition}
\newtheorem{definition}{Definition}
\newtheorem*{definition*}{Definition}
\newtheorem{example}{Example}
\newtheorem{remark}{Remark}

\def\PP{{\mathbb P}}

\def\KK{{\mathbb K}}

\def\GG{{\mathbb G}}

\def\CC{{\mathbb C}}

\def\gg{\mathfrak{g}}
\def\gl{\mathfrak{gl}}
\def\mm{\mathfrak{m}}
\def\Mat{{\rm Mat}}

\def\Ker{\mathop{\rm Ker}}

\def\Aut{\mathop{\rm Aut}}

\def\GL{\mathop{\rm GL}}

\def\Cl{\mathop{\rm Cl}}

\newcommand{\CO}{{\mathcal{O}}}

\begin{document}
\sloppy
\title[Additive actions on hypersurfaces]
{Additive actions on projective hypersurfaces}
\author{Ivan Arzhantsev}
\thanks{Both authors were partially supported by the Ministry of Education
and Science of Russian Federation, project 8214. The first author was
supported by the Dynasty Foundation, the Simons-IUM Grant, and the RFBR grant
12-01-00704.}
\address{I.A. \ Department of Higher Algebra, Faculty of Mechanics and Mathematics,
Moscow State University, Leninskie Gory 1, GSP-1, Moscow, 119991, Russia}
\address{
National Research University Higher School of Economics,
School of Applied Mathematics and Information Science,
Bolshoi Trekhsvyatitelskiy~3, Moscow, 109028, Russia}
\email{arjantse@mccme.ru}
\author{Andrey Popovskiy}
\address{A.P.\ Department of Higher Algebra, Faculty of Mechanics and Mathematics,
Moscow State University, Leninskie Gory 1, GSP-1, Moscow, 119991,
Russia } \email{PopovskiyA@gmail.com}
\date{\today}
\begin{abstract}
By an additive action on a hypersurface $H$ in $\PP^{n+1}$ we mean an
effective action of a commutative unipotent group on $\PP^{n+1}$ which
leaves $H$ invariant and acts on $H$ with an open orbit. Brendan Hassett and Yuri Tschinkel
have shown that actions of commutative unipotent groups on projective spaces
can be described in terms of local algebras with some
additional data. We prove that additive actions on projective hypersurfaces
correspond to invariant multilinear symmetric forms on local algebras.
It allows us to obtain explicit classification results for non-degenerate
quadrics and quadrics of corank one.
\end{abstract}
\maketitle
\section*{Introduction}

Let $\KK$ be an algebraically closed field of characteristic zero and
$\GG_a$ the additive group $(\KK,+)$. Consider the commutative unipotent
affine algebraic group $\GG_a^n$. In other words, $\GG_a^n$ is the additive
group of an $n$-dimensional vector space over $\KK$. The first aim of this paper is to survey
recent results on actions of $\GG_a^n$ with an open orbit on projective algebraic
varieties. To this end we include a detailed proof of the
Hassett-Tschinkel correspondence, discuss its corollaries, interpretations,
and related examples. Also we develop the method of Hassett and Tschinkel to
show that the generically transitive actions of the group $\GG_a^n$ on projective hypersurfaces
correspond to invariant multilinear symmetric forms on finite-dimensional local algebras.
This leads to explicit classification results for non-degenerate quadrics and quadrics of corank one.

By an additive action on a variety $X$ we mean a faithful regular action
of the group $\GG_a^n$ on $X$ such that one of the orbits is open in $X$.
The study of such actions was initiated by Brendan Hassett and Yuri Tschinkel~\cite{HT}.
They showed that additive actions on the projective space~$\PP^n$ up to equivalence are
in bijection with isomorphism classes of local algebras of $\KK$-dimension~$n+1$.
In particular, the number of additive actions on $\PP^n$ is finite if and
only if $n\le 5$.

Additive actions on projective subvarieties $X\subseteq\PP^m$ induced by
an action $\GG_a^n\times\PP^m\to\PP^m$ can be described in terms of local
$(m+1)$-dimensional algebras equipped with some additional data. This approach
was used in \cite{Sh} and \cite{ASh} to classify additive actions on
projective quadrics. Elena Sharoiko proved in \cite[Theorem~4]{Sh} that an additive
action on a non-degenerate quadric $Q\subseteq\PP^{n+1}$ is unique up to
equivalence. Recently Baohua Fu and Jun-Muk Hwang~\cite{FH} used a different
technique to show the uniqueness of additive action on a class of Fano
varieties with Picard number $1$. This result shows that
an abundance of additive actions on the projective space should be considered as
an exception.

A variety with a given additive action looks like an "additive analogue"
of a toric variety. Unfortunately, it turns out that two theories have almost
no parallels, see \cite{HT},~\cite{ASh}.

Generalized flag varieties $G/P$ of a semisimple algebraic group $G$ admitting an additive action are classified in \cite{Ar}. Roughly speaking, an additive action on $G/P$
exists if and only if the unipotent radical $P^u$ of the parabolic subgroup $P$ is commutative.
The uniqueness result in this case follows from~\cite{FH}. In particular, it covers the case of Grassmannians and thus answers a question posed in~\cite{ASh}. Another proof
of the uniqueness of additive actions on generalized flag varieties is obtained by
Rostislav Devyatov~\cite{Dev}. It uses nilpotent multiplications on certain finite-dimensional
modules over semisimple Lie algebras.

Evgeny Feigin proposed a construction based on the PBW-filtration to degenerate an arbitrary
generalized flag variety $G/P$ to a variety with an additive action, see \cite{Fe} and
further publications.

In~\cite{DL}, Ulrich Derenthal and Daniel Loughran classified singular del Pezzo surfaces
with additive actions; see also~\cite{DL2}. By the results of~\cite{CLT},
Manin's Conjecture is true for such surfaces.

In this paper we prove that additive actions on projective hypersurfaces of degree $d$
in $\PP^{n+1}$ are in bijection with invariant $d$-linear symmetric forms on
$(n+2)$-dimensional local algebras. The corresponding form is the polarization of
the equation defining the hypersurface. As an application, we give a short proof
of uniqueness of additive action on non-degenerate quadrics and classify
additive actions on quadrics of corank one. The case of cubic projective
hypersurfaces is studied in the recent preprint of Ivan Bazhov~\cite{Ba}.

The paper is orzanized as follows. In Section~\ref{sec1} we define additive actions
and consider the problem of extension of an action $\GG_a^n\times X\to X$ on
a projective hypersurface $X$ to the ambient space $\PP^{n+1}$.
The Hassett-Tschinkel correspondence is discussed in Section~\ref{sec2}.
Section~\ref{sec3} is devoted to invariant multilinear symmetric forms on
local algebras. Our main result (Theorem~\ref{T4}) describes
additive actions on projective hypersurfaces in these terms. Also
we give an explicit formula for an invariant multilinear symmetric form (Lemma~\ref{lem})
and prove that if a hypersurface $X$ in $\PP^{n+1}$ admits an additive action and
the group $\Aut(X)^0$ is reductive, then $X$ is either a hyperplane or
a non-degenerate quadric (Proposition~\ref{last}).
Additive actions on non-degenerate quadrics and
on quadrics of corank one are classified in Section~\ref{sec4} and Section~\ref{sec5}, respectively.

\section{Additive actions on projective varieties}
\label{sec1}

Let $X$ be an irreducible algebraic variety of dimension $n$ and $\GG_a^n$
be the commutative unipotent group.

\begin{definition}
An \emph{inner additive action} on $X$ is an effective action
$\GG_a^n\times X\to X$ with an open orbit.
\end{definition}

It is well known that for an action
of a unipotent group on an affine variety all orbits are closed,
see, e.g., \cite[Section~1.3]{PV}. It implies that if an affine variety $X$
admits an additive action, then $X$ is isomorphic to the group $\GG_a^n$
with the $\GG_a^n$-action by left translations.

In general, the existence of an inner additive action implies that the variety $X$
is rational. For $X$ normal, the divisor class group $\Cl(X)$ is freely generated
by prime divisors in the complement of the open $\GG_a^n$-orbit. In particular,
$\Cl(X)$ is a free finitely generated abelian group.

The most interesting case is the study of
inner additive actions on complete varieties~$X$. In this case an inner
additive action determines a maximal commutative unipotent subgroup of the linear algebraic group $\Aut(X)^0$. Two inner additive actions are said to be \emph{equivalent},
if the corresponding subgroups are conjugate in $\Aut(X)$.

\begin{proposition} \label{spher}
Let $X$ be a complete variety with an inner additive action.
Assume that the group $\Aut(X)^0$ is reductive. Then $X$ is a generalized
flag variety $G/P$, where $G$ is a linear semisimple group and $P$ is a parabolic
subgroup.
\end{proposition}

\begin{proof}
Let $X'$ be the normalization of $X$. Then the action of $\Aut(X)^0$ lifts to $X'$.
By the assumption, some unipotent subgroup of $\Aut(X)^0$ acts on $X'$ with an open orbit.
Then a maximal unipotent subgroup of the reductive group $\Aut(X)^0$ acts on $X'$
with an open orbit. It means that $X'$ is a spherical variety of rank zero,
see~\cite[Section~1.5.1]{Ti} for details.
It yields that $X'$ is a generalized flag variety $G/P$, see
\cite[Proposition~10.1]{Ti}, and $\Aut(X)^0$ acts on $X'$ transitively.
The last condition implies that $X=X'$.
\end{proof}

A classification of generalized flag varieties admitting an inner additive action
is obtained in~\cite{Ar}. In particular, the parabolic subgroup $P$ is maximal in this case.

\begin{definition}
Let $X$ be a closed subvariety of dimension $n$ in the projective space $\PP^m$. Then
an \emph{additive action} on $X$ is an effective action
$\GG_a^n\times\PP^m\to\PP^m$ such that $X$ is $\GG_a^n$-invariant
and the induced action $\GG_a^n\times X\to X$ has an open orbit.
Two additive actions on $X$ are said to be \emph{equivalent}
if one is obtained from another via automorphism of $\PP^m$ preserving $X$.
\end{definition}

Clearly, any additive action on a projective subvariety $X$ induces an inner
additive action on $X$. The converse is not true, i.e., not any action
$\GG_a^n \times X\to X$ with an open orbit on a projective subvariety $X$
can be extended to the ambient space $\PP^m$.

\begin{example}
Consider a subvariety
$$
X=V(x^2z-y^3)\subseteq \PP^2
$$
and a rational $\GG_a$-action on $X$ given by
$$
\left(\frac{y}{x}, \ a\right) \mapsto \frac{y}{x} + a.
$$
Using affine charts one can check that this action is regular.
On the other hand, it can not be extended to $\PP^2$,
because the closure of a $\GG_a$-orbit on $\PP^2$ can not be
a cubic, see Example~\ref{gap2}.
\end{example}

At the same time, if the subvariety $X$ is linearly normal in $\PP^m$ and $X$ is normal,
then an extension of a $\GG_a^n$-action to $\PP^m$ exists. Indeed, the restriction
$L=:\CO(1)|_X$ of the line bundle $\CO(1)$ on $\PP^m$ can be linearized with respect
to the action $\GG_a^n\times X\to X$, see e.g. \cite{KKLV}. The~linearization defines
a structure of a rational $\GG_a^n$-module on the space of sections $H^0(X,L)$.
Since $X$ is linearly normal, the restriction $H^0(\PP^m,\CO(1))\to H^0(X,L)$
is surjective. Consider a vector space decomposition $H^0(\PP^m,\CO(1))=V_1\oplus V_2$,
where $V_1$ is the kernel of the restriction. The complementary subspace $V_2$ projects
to $H^0(X,L)$ isomorphically. This isomorphism induces a structure
of a rational $\GG_a^n$-module on $V_2$. Further,
we regard $V_1$ as the trivial $\GG_a^n$-module.
This gives a structure of a rational $\GG_a^n$-module on
$H^0(\PP^m,\CO(1))$. Since $\PP^m$ is the projectivization of the dual space
to $H^0(\PP^m,\CO(1))$, we obtain a required extended action $\GG_a^n\times \PP^m\to \PP^m$.

\smallskip

From now on we consider additive actions on projective subvarieties $X\subseteq\PP^m$.


\section{The Hassett-Tschinkel correspondence}
\label{sec2}

In~\cite{HT} Brendan Hassett and Yuri Tschinkel established
a remarkable correspondence between additive actions on the projective space
$\PP^n$ and local algebras of $\KK$-dimension $n+1$. Moreover, they
described rational cyclic $\GG_a^n$-modules in terms of local algebras.
In this section we recall these results. The proofs given here are taken from~\cite{ASh}.
By a local algebra we always mean a commutative associative local algebra
with unit.

Let $\rho:\GG_a^n\to\GL_{m+1}(\KK)$ be a faithful rational representation.
The differential defines a representation $d\rho:\gg\to\gl_{m+1}(\KK)$ of the tangent algebra $\gg = \text{Lie}(\GG_a^n)$ and the induced representation $\tau:U(\gg)\to\Mat_{m+1}(\KK)$ of the universal enveloping algebra $U(\gg)$. Since the group $\GG_a^n$ is commutative, the algebra $U(\gg)$ is isomorphic to the polynomial algebra
$\KK[x_1,\dots,x_n]$, where  $\gg$ is identified with the subspace $\langle x_1,\dots,x_n\rangle$.
The algebra $R:=\tau(U(\gg))$ is isomorphic to the factor algebra $U(\gg)/\Ker\tau$.
As $\tau(x_1),\ldots,\tau(x_n)$ are commuting
nilpotent operators, the algebra $R$ is finite-dimensional and local. Let us denote by $X_1,\dots, X_n$ the images of the elements $x_1,\dots,x_n$
in~$R$. Then the maximal ideal of $R$ is $\mm:=(X_1,\dots, X_n)$.
The subspace $W:=\tau(\gg)=\langle X_1,\ldots,X_n\rangle$ generates $R$
as an algebra with unit.

Assume that $\KK^{m+1}$ is a cyclic $\GG_a^n$-module with a cyclic vector $v$, i.e., $\langle\rho(\GG_a^n)v\rangle=\KK^{m+1}$.
The subspace $\tau(U(\gg))v$ is $\gg$-  and $\GG_a^n$-invariant; it contains the vector $v$ and therefore coincides with the space $\KK^{m+1}$. Let $I=\{y\in U(\gg)\, : \, \tau(y)v=0\}$. Since the vector $v$ is cyclic, the ideal $I$ coincides with
$\Ker\tau$, and we obtain identifications
$$
R \cong U(\gg)/I \cong \tau(U(\gg))v=\KK^{m+1}.
$$
Under these identifications the action of an element $\tau(y)$ on $\KK^{m+1}$ corresponds to the operator of multiplication by $\tau(y)$ on the factor algebra $R$, and the vector $v\in\KK^{m+1}$ goes to the residue class of unit. Since
$\GG_a^n = \exp(\gg)$, the $\GG_a^n$-action on $\KK^{m+1}$ corresponds to the multiplication by elements of $\exp(W)$ on $R$.

\smallskip

Conversely, let $R$ be a local $(m+1)$-dimensional algebra with a maximal ideal $\mm$, and $W\subseteq \mm$ be a subspace that
generates $R$ as an algebra with unit.
Fix a basis $X_1,\ldots,X_n$ in $W$. Then $R$ admits a presentation $\KK[x_1,\ldots,x_n]/I$,
where $I$ is the kernel of the homomorphism
$$
\KK[x_1,\ldots,x_n]\to R, \, x_i \mapsto X_i.
$$
These data define a faithful representation $\rho$ of the group $\GG_a^n:=\exp(W)$ on the space $R$: the operator $\rho((a_1,\dots,a_n))$ acts as multiplication by the element $\exp(a_1X_1+\dots+a_nX_n)$. Since $W$ generates $R$ as an algebra with unit,
one checks that the representation
is cyclic with unit in $R$ as a cyclic vector.

Summarizing, we obtain the following result.

\begin{theorem} \cite[Theorem~2.14]{HT}. \label{T1}
The correspondence described above establishes a bijection between
\begin{enumerate}
\item[$(1)$] equivalence classes of faithful cyclic rational representations $\rho:\GG_a^n\to\GL_{m+1}(\KK)$;
\item[$(2)$] isomorphism classes of pairs $(R,W)$, where $R$ is a local $(m+1)$-dimensional algebra
with the maximal ideal $\mm$ and $W$ is an $n$-dimensional subspace of $\mm$ that generates~$R$ as an algebra with unit.
\end{enumerate}
\end{theorem}

\begin{remark} \label{r1}
Let $\rho:\GG_a^n\to\GL_{m+1}(\KK)$ be a faithful cyclic rational representation.
The set of cyclic vectors in $\KK^{m+1}$ is an open orbit of a commutative algebraic
group $C$ with  $\rho(\GG_a^n)\subseteq C \subseteq \GL_{m+1}(\KK)$, and the complement
of this set is a hyperplane. In our notation, the group $C$ is the extension
of the commutative unipotent group $\exp(\mm)\cong\GG_a^m$ by scalar matrices.
\end{remark}

A faithful linear representation $\rho:\GG_a^n\to\GL_{m+1}(\KK)$
determines an effective action of the group $\GG_a^n$ on the projectivization $\PP^m$ of the space $\KK^{m+1}$.
Conversely, let $G$ be a connected affine algebraic group with the trivial Picard group, and $X$ be a normal
$G$-variety. By \cite[Section~2.4]{KKLV}, every line bundle on $X$ admits a $G$-linearization.
Moreover, if $G$ has no non-trivial characters, then a $G$-linearization is unique.
This shows that every effective $\GG_a^n$-action on $\PP^m$ comes from a (unique) faithful rational $(m+1)$-dimensional $\GG_a^n$-module.

An effective $\GG_a^n$-action on $\PP^m$ has an open orbit if and only if $n=m$.
In this case the corresponding $\GG_a^n$-module is cyclic.
It terms of Theorem~\ref{T1} the condition $n=m$ means $W=\mm$, and we obtain the following theorem.

\begin{proposition}\cite[Proposition~2.15]{HT} \label{T2}
There is a one-to-one correspondence between
\begin{enumerate}
\item[$(1)$] equivalence classes of additive actions on $\PP^n$;
\item[$(2)$] isomorphism classes of local $(n+1)$-dimensional algebras.
\end{enumerate}
\end{proposition}

\begin{remark}
It follows from Remark~\ref{r1} that if the group $\GG_a^n$ acts on $\PP^m$ and
some orbit is not contained in a hyperplane, then the action
can be extended to an additive action $\GG_a^m\times\PP^m\to\PP^m$.
It seems that such an extension exists without any extra assumption.
\end{remark}

Given the projectivization $\PP^m$ of a faithful rational $\GG_a^n$-module and a point $x\in\PP^m$ with the trivial stabilizer,
the closure $X$ of the orbit $\GG_a^n\cdot x$ is a projective variety equipped with an additive
action. Closures of generic orbits are hypersurfaces if and only if $n=m-1$. If such a hypersurface is not a hyperplane, then $\PP^m$ comes from the projectivization of a cyclic $\GG_a^n$-module, it is given by a pair $(R,W)$, and the condition $n=m-1$ means that $W$
is a hyperplane in $\mm$. We obtain the following result.

\begin{proposition} \label{T3}
There is a one-to-one correspondence between
\begin{enumerate}
\item[$(1)$] equivalence classes of additive actions on hypersurfaces in $\PP^{n+1}$
of degree at least two;
\item[$(2)$] isomorphism classes of pairs $(R,W)$, where $R$ is a local $(n+2)$-dimensional algebra with the maximal ideal $\mm$ and $W$ is a hyperplane in $\mm$ that generates~$R$ as an algebra with unit.
\end{enumerate}
\end{proposition}

It is shown in \cite[Theorem~5.1]{ASh} that the degree of the hypersurface corresponding
to a pair $(R,W)$ is the maximal exponent $d$ such that the subspace $W$
does not contain the ideal $\mm^d$.

\begin{example} \label{gap2}
There exist two 3-dimensional local algebras,
$$
\KK[x]/(x^3) \quad \text{and} \quad \KK[x,y]/(x^2,xy,y^2).
$$
In the first case $\mm^3=0$, and in the second one we have $\mm^2=0$.
This shows that for every $\GG_a$-action on $\PP^2$ the orbit closures are
either lines or quadrics.
\end{example}


\section{Invariant multilinear forms on local algebras}
\label{sec3}

Consider a pair $(R,W)$ as in Proposition~\ref{T3} and let $H\subseteq\PP^{n+1}$
be the corresponding hypersurface. Let us fix a coordinate system on
$R=\langle 1\rangle\oplus\mm$ such that $x_0$ is the coordinate along $\langle 1\rangle$
and $x_1,\ldots,x_{n+1}$ are coordinates on $\mm$.

Assume that $H$ is defined by a homogeneous equation
$$
f(x_0,x_1,\ldots,x_{n+1})=0
$$
of degree $d$. Since $H$ is invariant under the action of $\GG_a^n$,
the polynomial $f$ is $\GG_a^n$-semi-invariant \cite[Theorem~3.1]{PV}.
But the group $\GG_a^n$ has no non-trivial characters, and the polynomial
$f$ is $\GG_a^n$-invariant. Equivalently, $f$ is annihilated
by the Lie algebra $\gg$.

It is well known that for a given homogeneous polynomial $f$ of degree $d$
on a vector space $R$ there exists a unique $d$-linear symmetric map
$$
F \colon R\times R\times\ldots\times R\to \KK
$$
such that $f(v)=F(v,v,\ldots,v)$ for all $v\in R$, see e.g. \cite[Section~9.1]{PV}.
The map $F$ is called the \emph{polarization} of the polynomial $f$.

Since the representation $d\rho$ of the Lie algebra $\gg$ on $R$ is given
by multiplication by elements of $W,$ a homogeneous polynomial $f$ on $R$
is annihilated by $\gg$ if and only if
\begin{equation}
\label{star}
F(ab_1,b_2,\ldots,b_d)+F(b_1,ab_2,\ldots,b_d)+\ldots+F(b_1,b_2,\ldots,ab_d)=0
\ \ \forall \ a\in W, b_1,\ldots,b_d\in R.
\end{equation}

\begin{definition} \label{def}
Let $R$ be a local algebra with the maximal ideal $\mm$.
An \emph{invariant $d$-linear form} on $R$ is a $d$-linear symmetric map
$$
F \colon R\times R\times\ldots\times R\to \KK
$$
such that $F(1,1,\ldots,1)=0$, the restriction of $F$ to $\mm$ is nonzero,
and there exists a hyperplane $W$ in $\mm$
which generates $R$ as an algebra with unit and such that condition~(\ref{star}) holds.
\end{definition}

If $F_1$ (resp. $F_2$) are invariant $d_1$-linear (resp. $d_2$-linear)
forms on $R$ with respect to the same hyperplane $W$, then the product $F_1F_2$ defines
an invariant $(d_1+d_2)$-linear form. An invariant multilinear form is said to be
\emph{irreducible}, if it can not be represented as such a product.

One can show that there is no invariant linear form. It implies that any
invariant bilinear or $3$-linear form is irreducible.

\smallskip

We are ready to formulate our main result.

\begin{theorem} \label{T4}
Additive actions on hypersurfaces of degree $d\ge 2$  in $\PP^{n+1}$
are in natural one-to-one correspondence with pairs $(R,F)$, where $R$
is a local algebra of dimension $n+2$ and $F$ is an irreducible invariant
$d$-linear form on $R$ up to a scalar.
\end{theorem}

\begin{proof}
An additive action on a hypersurface $H$ in $\PP^{n+1}$ is given
by a faithful rational representation $\rho:\GG_a^n\to\GL_{n+2}(\KK)$
making $\KK^{n+2}$ a cyclic $\GG_a^n$-module. In our correspondence we
identify $\KK^{n+2}$ with the local algebra $R$. We choose coordinates
$x_0,x_1,\ldots,x_{n+1}$ compatible with the decomposition
$R=\langle 1\rangle\oplus\mm$. Let $f(x_0,x_1,\ldots,x_{n+1})=0$ be the equation
of the hypersurface $H$, where $f$ is irreducible. Then the algebra of $\GG_a^n$-invariants
on $\KK^{n+2}$ is freely generated by $x_0$ and $f$.

Every $\GG_a^n$-invariant hypersurface in $\PP^{n+1}$ is given by
$$
\alpha f(x_0,x_1,\ldots,x_{n+1})+\beta x_0^d=0,
\quad (\alpha,\beta)\in\KK^2\setminus\{(0,0)\}.
$$
So we may assume that $f$ does not contain the term $x_0^d$. Let $F$ be the polarization
of $f$. Then condition $(\ref{star})$ holds and $F(1,\ldots,1)=0$. If the restriction of $F$ to
$\mm$ is zero, then $x_0$ divides~$f$, and $f$ is not irreducible, a contradiction.

Conversely, let $(R,F)$ be such a pair and $W$ be a subspace from the definition of $F$.
Then $(R,W)$ gives rise to a structure of a rational $\GG_a^n$-module on $R$.
Consider the hypersurface $f=0$ in $\PP(R)\cong\PP^{n+1}$, where $f(v)=F(v,v,\ldots,v)$.
It is invariant under the action of $\GG_a^n$. By the assumptions, $f$ is irreducible and thus
the hypersurface $f=0$ coincides with the closure of a generic $\GG_a^n$-orbit.
\end{proof}

Given a hyperplane $W$ in the maximal ideal $\mm$ of a local algebra $R$ that generates $R$
as an algebra with unit,
let $d$ be the maximal exponent such that the subspace $W$ does not contain the ideal $\mm^d$.
By Theorem~\ref{T4} and~\cite[Theorem~5.1]{ASh}, there exists a unique up to
a scalar irreducible invariant (with respect to $W$) $d$-linear form $F_W$ on $R$.
Let us write down this form explicitly.

By linearity, we may assume that each argument of $F_W$ is either the unit $1$ or
an element of $\mm$. Fix an isomorphism $\mm/W\cong\KK$ and consider the projection $\pi\colon\mm\to\mm/W\cong\KK$.
We define the form
$$
F_W(b_1,\ldots,b_d):=(-1)^k k!(d-k-1)!\,\pi(b_1\ldots b_d),
$$
where $k$ is the number of units among $b_1,\ldots,b_d$, and for $k=d$ we let
$F_W(1,\ldots,1)=0$.

\begin{lemma} \label{lem}
The form $F_W$ is an irreducible invariant $d$-linear form on $R$.
\end{lemma}

\begin{proof}
We begin with condition (\ref{star}). Since $ab_i\in\mm$ for all $a\in W$ and $b_i\in R$,
with $0<k<d$ this condition can be rewritten as
$$
(k(-1)^{k-1}(k-1)!(d-k)!+(d-k)(-1)^k k!(d-k-1)!)\,\pi(ab_1\ldots b_d)=0,
$$
and it is obvious. For $k=0$ we have $ab_1\ldots b_d\in\mm^{d+1}\subseteq W$,
and thus $\pi(ab_1\ldots b_d)=0$. For $k=d$ we have $-(d-2)!\pi(a)=0$, because $a\in W$.

The restriction of $F$ to $\mm$ is nonzero since $\mm^d$ is not contained in $W$.
It follows from \cite[Theorem~5.1]{ASh} that the form $F$ is irreducible.
Finally, we have $F_W(1,\ldots,1)=0$ by definition.
\end{proof}

The next proposition follows immediately from \cite[Theorems~1 and 4]{MM}.
Let us obtain this result using our technique.

\begin{proposition} \label{smooth}
Let $H$ be a smooth hypersurface in $\PP^{n+1}$ admitting an additive action.
Then $H$ is either a hyperplane or a non-degenerate quadric.
\end{proposition}

\begin{proof}
Assume that an additive action on $H$ is given by a triple $(R,W,F)$. Let
$e_0,e_1,\ldots,e_{n+1}$ be a basis of $R$ compatible with the decomposition
$$
R=\langle 1\rangle\oplus W\oplus\langle e_{n+1}\rangle.
$$
Moreover, we may assume that $e_{n+1}\in\mm^d$, where $\mm^{d+1}$ is contained in $W$. Then
in the notation of Lemma~\ref{lem} we have $\pi(be_{n+1})=0$ for all $b\in\mm$.
It means that the variable $x_{n+1}$ can appear in the equation $f(x_0,\ldots,x_{n+1})=0$
of the hyperplane $H$ only in the term $x_0^{d-1}x_{n+1}$. Thus the point
$[0:\ldots:0:1]$ lies on $H$ and it is singular provided $d\ge 3$.
It remains to note that the only smooth quadric is a non-degenerate one.
\end{proof}

\begin{proposition} \label{last}
Let $H$ be a hypersurface in $\PP^{n+1}$ which admits an additive action and
such that the group $\Aut(H)^0$ is reductive. Then $H$ is either a hyperplane or
a non-degenerate quadric.
\end{proposition}

\begin{proof}
By Proposition~\ref{spher}, the variety $H$ is smooth, and the assertion follows
from Proposition~\ref{smooth}.
\end{proof}

\begin{remark}
Take a triple $(R,W,F)$ as in Definition~\ref{def} and consider the sum $I$ of all
ideals of the algebra $R$ contained in $W$. It is the biggest ideal of $R$
contained in $W$. Taking a compatible basis of $R$, we see that the equation
of the corresponding hypersurface does not depend on the coordinates in $I$.
Moreover, for the factor algebra $R'=R/I$ we have
$$
R'=\langle 1'\rangle \oplus W'\oplus(\mm')^d
$$
with $\mm'=\mm/I$, $W'=W/I$, and $\dim(\mm')^d=1$. The invariant form $F$
descents to $R'$, the subspace $W'$ contains no ideal of $R'$, and the algebra $R'$
is Gorenstein. Such a reduction is useful in classification problems.
\end{remark}


\section{Non-degenerate quadrics}
\label{sec4}

In this section we classify non-degenerate invariant bilinear symmetric forms on local
algebras. These results are obtained in~\cite{Sh}, but we give a short elementary proof.

Let $R$ be a local algebra of dimension $n+2$ with the maximal ideal $\mathfrak m$  and
$F$ a non-degenerate bilinear symmetric form on $R$ such that
$F(1,1)=0$. Assume that for some hyperplane $W$ in $\mathfrak m$ generating $R$ we have
\begin{equation}
\label{star2}
F(ab_1, b_2) + F(b_1, ab_2) = 0 \quad \mbox{for all } b_1,b_2\in R \mbox { and } a\in W.
\end{equation}

We choose a basis $e_0=1,e_1,\ldots,e_n,e_{n+1}$ of $R$ such that
$W=\langle e_1,\ldots,e_n\rangle$ and $\mm=\langle e_1,\ldots,e_{n+1}\rangle$.
For any $b\in R$ let $b=b^{(0)}+b^{(1)}+\ldots+b^{(n+1)}$ be the decomposition
corresponding to this basis.

\begin{lemma}
\label{tl}
\begin{enumerate}
\item[1)]
$F(1,a)=0$ for all $a\in W$;
\item[2)]
$F(1,b)=F(1,b^{(n+1)})$ for all $b\in R$;
\item[3)]
If $a,a'\in W$ and $aa'\in W$, then $aa'=0$;
\item[4)]
The restriction of the form $F$ to $W$ is non-degenerate.
\end{enumerate}
\end{lemma}

\begin{proof}
Assertion~1) follows from~(\ref{star2}) with $b_1=b_2=1$. For~2), note that
$$
F(1,b)=F(1,b^{(0)})+F(1,b^{(1)}+\ldots+b^{(n)})+F(1,b^{(n+1)}).
$$
The first term is $0$ because of $F(1,1)=0$, and the second one is $0$ by~1).
If $a,a',aa'\in W$, then for any $b\in R$ we have
$$
F(b,aa')=-F(ab,a')=F(aa'b,1)=-F(b,aa'),
$$
and $F(b,aa')=0$. Since $F$ is non-degenerate, we obtain 3).

For 4), assume that for some $0\ne a\in W$ we have $F(a,a')=0$ for all $a'\in W$.
Since $F$ is non-degenerate, it yields $F(a,e_{n+1})=\lambda$ for some nonzero $\lambda\in\KK$.
If $F(1,e_{n+1})=\mu$, then $F(\mu a-\lambda 1, e_{n+1})=0$, and the vector $\mu a-\lambda 1$
is in the kernel of the form $F$.
\end{proof}

Let us denote by $M(F)$ the matrix of a bilinear form $F$ in a given basis.

\begin{proposition}
		\label {th1}
In the notation as above, the triple $(R,W,F)$ can be transformed into the form
$$
R=\KK[e_1,\dots,e_n]/(e_i^2 - e_j^2, e_ie_j; \ 1\le i<j\le n),
\qquad
W=\langle e_1,\dots,e_n\rangle,
$$
$$M(F) =
\begin{pmatrix}
				0 			&0 			&\ldots	&0	&1\\
				0 			&-1			&\ldots	&0 	&0\\
				\vdots 	&\vdots			&\ddots	&\vdots	&\vdots\\
				0 	    &0 	&\ldots &-1 				&0\\
				1			  &0			&\ldots			&0	&0\\
\end{pmatrix}
$$		
\end{proposition}
		
\begin{proof}
Since $F$ is non-degenerate, we may assume that $F(1,e_{n+1})=1$.
Using Lemma~\ref{tl},~4), we may suppose that $F(e_i,e_j)=-\delta_{ij}$
for all $1\le i,j\le n$.
Now the matrix of the form $F$ looks like
$$
\begin{pmatrix}
				0 			&0 			&\ldots	&0	&1\\
				0 			&-1			&\ldots	&0 	&*\\
				\vdots 	&\vdots			&\ddots	&\vdots	&\vdots\\
				0 	    &0 	&\ldots &-1 				&*\\
				1			  &*			&\ldots			&*	&*\\
\end{pmatrix}
$$
		
For all $1\le i<j\le n$ we have $F(1,e_ie_j) = - F(e_i,e_j) = 0$.
It follows from Lemma~\ref{tl}, 2) that $(e_ie_j)^{(n+1)} = 0$.
We conclude that $e_ie_j\in W$ and thus $e_ie_j = 0$ by Lemma~\ref{tl}, 3).
		
Since $F(1, e_i^2) = -F(e_i, e_i) = 1$, we have
$e_i^2 = e_{n+1} + f_i$, where $f_i \in W$.
Then
$$
(e_1 + e_i)(e_1 - e_i) = e_1^2 - e_i^2 = f_1-f_i \in W.
$$
By Lemma~\ref{tl}, 3) we obtain $e_1^2 = e_i^2$.
		
Without loss of generality it can be assumed that $e_{n+1}=e_1^2$.
Let $n\ge 2$. Then $e_{n+1}e_i = e_j^2e_i = 0$, where $1\le i\ne j\le n$.
If $n = 1$ then $e_{n+1}e_1 = e_1^3 \in \mm^3=0$.

Hence $e_{n+1}b=0$ for any element $b\in \mm$, and $R$ is isomorphic to
$\KK[e_1,\dots,e_n]/(e_i^2-e_j^2, e_ie_j)$.
		
It remains to prove that $e_{n+1}\perp_F\! \mm$.
Indeed, we have
$$
F(e_{n+1},b)=F(e_1^2,b)=-F(e_1,e_1b)=F(1,e_1^2b)=F(1,0)= 0 \quad \forall b\in\mm.
$$
This completes the proof of the proposition.
\end{proof}

As a corollary we obtain the result of \cite[Theorem~4]{Sh}.

\begin{corollary}
A non-degenerate quadric $Q_n\subseteq\PP^{n+1}$ admits a unique additive
action up to equivalence.
\end{corollary}


\section{Quadrics of corank one}
\label{sec5}

Let us classify invariant bilinear symmetric forms of rank $n+1$ on local
$(n+2)$-dimensional algebras. Geometrically these results can be interpreted
as a classification of additive actions on quadrics of corank one in $\PP^{n+1}$.

Let $R$ be a local algebra of dimension $n+2$, $n\ge 2$, with the maximal ideal $\mathfrak m$ and
$F$ a bilinear symmetric form of rank $n+1$ on $R$ such that $F(1,1)=0$.
Assume that for some hyperplane
$W$ in $\mathfrak m$ condition~(\ref{star2}) holds.
We choose a basis $e_0=1,e_1,\ldots,e_n,e_{n+1}$ of $R$ such that
$W=\langle e_1,\ldots,e_n\rangle$ and $\mm=\langle e_1,\ldots,e_{n+1}\rangle$.

\begin{lemma}
\label{kerl}
The kernel $\Ker F$ is contained in $W$.
\end{lemma}

\begin{proof}
Let $\Ker F=\langle l\rangle$. Assume that $l$ is not in $W$. Then we should consider
four alternatives.

\begin{enumerate}
\item Let $\langle l\rangle = \langle 1\rangle$.
Then $F(a, b) = -F(1, ab) = 0$ for all $a\in W, b\in R$,
and $\dim\Ker F \ge 2$, a contradiction.
\item Let $\langle l\rangle \subseteq \mm\setminus W$.
Without loss of generality it can be assumed that $l=e_{n+1}$.
As we have seen, $F(1,b)=F(1,b^{(n+1)})=0$ for all $b\in R$.
Thus we have $1\in\Ker F$, which leads to a contradiction.
\item Let $\langle l\rangle\subseteq R \setminus
(\mm \cup \langle 1, W\rangle)$.
Without loss of generality it can be assumed that $l=1+e_{n+1}$.
We have $0=F(1,l)=F(1,1)+F(1,e_{n+1})=F(1,e_{n+1})$.
It again follows that $1\in\Ker F$.
\item Let $\langle l\rangle\subseteq \langle 1, W\rangle\setminus W$.
We can assume that $l=1+f$, where $W\ni f\ne 0$.
Then
$$
F(1,b)=-F(f,b)=F(1,fb)=\ldots=F(1,f^{n+3}b)=0 \quad \forall \, b\in R.
$$
Thus we again have $1\in\Ker F$.
\end{enumerate}

This completes the proof of the lemma.
\end{proof}

\begin{proposition}
\label {prco}
In the notation as above, the triple $(R,W,F)$ can be transformed into the form
$$
M(F)=
\begin{pmatrix}
				0 			&0 			&\ldots	&0	&0 &1\\
				0 			&-1			&\ldots	&0 	&0 &0\\
				\vdots 	&\vdots			&\ddots	&\vdots	&\vdots &\vdots\\
				0 	    &0 	&\ldots &-1 		&0	&0\\
				0       &0  &\ldots         &0  &0 &0\\
        1			  &0	&\ldots			&0	&0 &0\\
\end{pmatrix},
\quad W=\langle e_1,\ldots,e_n\rangle,
$$
and $R$ is isomorphic to one of the following algebras:
\begin {enumerate}
\item
$
\KK[e_1,\ldots,e_n]/(e_ie_j - \lambda_{ij}e_n, e_i^2-e_j^2-(\lambda_{ii}-\lambda_{jj})e_n,
e_se_n, \, 1\le i<j\le n-1, 1\le s\le n, n\ge 3)
$
\smallskip
\noindent where $\lambda_{ij}$ are elements of a symmetric block diagonal $(n-1)\times (n-1)$-matrix $\Lambda$ such that each block $\Lambda_k$ is
    $$
			\lambda_k
			\begin{pmatrix}
				1& 0& & 0\\
				0& \ddots&\ddots& \\
				& \ddots& \ddots& 0\\
				0& & 0& 1\\
			\end{pmatrix}
			+\cfrac{1}{2}
			\begin{pmatrix}
				0& 1& & 0\\
				1& \ddots&\ddots& \\
				& \ddots& \ddots& 1\\
				0& & 1& 0\\
			\end{pmatrix}
			+\cfrac{i}{2}
			\begin{pmatrix}
				0& & 1& 0\\
				& \ddots& \ddots& -1\\
				1 & \ddots& \ddots&\\
				0& -1 & & 0\\
			\end{pmatrix}
		$$ 	
with some $\lambda_k\in\KK$;
\smallskip
		\item $\KK[e_1,e_2]/(e_1^3, e_1e_2,e_2^2)$  \ or \
          $\KK[e_1]/(e_1^4)$ with $e_2=e_1^3, e_3=e_1^2$.
		\end {enumerate}
\end{proposition}

\begin{remark}
Blocks $\Lambda_k$ of size $1$ are $\begin{pmatrix}\lambda_k\\ \end{pmatrix}$.
Blocks $\Lambda_k$ of size $2$ are
		$$
			\begin{pmatrix}
				\lambda_k + \cfrac{i}{2} & \cfrac{1}{2}\\
				\cfrac{1}{2} & \lambda_k - \cfrac{i}{2}\\
			\end{pmatrix}
		$$
	\end{remark}

\begin{proof}[Proof of Proposition~\ref{prco}]	
By Lemma~\ref{kerl} we may assume that $\Ker F=\langle e_n\rangle$
and $F(1,e_{n+1})=1$, because of $F(1,a)=0$ for all $a\in W$.
Let $V$ be the linear span $\langle e_1,\ldots,e_{n-1}\rangle$.
As in Lemma~\ref{tl}, 4) one can show that the restriction of $F$ to $V$		
is non-degenerate. Thus we can assume that the matrix of $F$ has the form
   $$
 		\begin{pmatrix}
				0 			&0 			&\ldots	&0	&0 &1\\
				0 			&-1			&\ldots	&0 	&0 &*\\
				\vdots 	&\vdots			&\ddots	&\vdots	&\vdots &\vdots\\
				0 	    &0 	&\ldots &-1 		&0	&*\\
				0       &0  &\ldots         &0  &0 &*\\
        1			  &*	&\ldots			&*	&* &*\\
    \end{pmatrix}.
	 	$$
We have $1 = -F(e_i,e_i) = F(1,e_i^2) = F(1,(e_i^2)^{(n+1)}) \Rightarrow e_i^2=e_{n+1}+f_i$ for some $f_i\in W$ and every $i=1,\ldots,n-1$. We may assume that $e_{n+1}=e_1^2$.

Then $f_i = e_i^2 - e_1^2 = (e_i + e_1)(e_i - e_1)$ and, as in Lemma~\ref{tl}, 3),
we obtain $f_i = \lambda_{ii}e_n$.
	 	
Again as in Lemma~\ref{tl}, 3), we have $e_ie_j = \lambda_{ij}e_n$ for all $1\le i<j\le n-1$.

Thus multiplication on the subspace $V$ is given by the matrix $I_ne_{n+1} + \Lambda e_n$,
where $I_n$ is the identity matrix
and a symmetric matrix $\Lambda$ is defined up to adding a scalar matrix.

It is easy to check that the symmetric matrix $\Lambda=(\lambda_{ij})$ under orthogonal
transformations on $V$ transforms as the matrix of a bilinear symmetric form.
It follows from \cite[Chapter~XI,~\S 3]{Gant} that $\Lambda$ can be
transformed into the canonical block diagonal form by orthogonal transformation.
Here each block $\Lambda_k$ has the form
		$$
			\lambda_k
			\begin{pmatrix}
				1& 0& & 0\\
				0& \ddots&\ddots& \\
				& \ddots& \ddots& 0\\
				0& & 0& 1\\
			\end{pmatrix}
			+\cfrac{1}{2}
			\begin{pmatrix}
				0& 1& & 0\\
				1& \ddots&\ddots& \\
				& \ddots& \ddots& 1\\
				0& & 1& 0\\
			\end{pmatrix}
			+\cfrac{i}{2}
			\begin{pmatrix}
				0& & 1& 0\\
				& \ddots& \ddots& -1\\
				1 & \ddots& \ddots&\\
				0& -1 & & 0\\
			\end{pmatrix},
\quad \lambda_k\in\KK.
		$$
	
We claim that $e_n\mm=0$. Indeed, $F(ae_n,b) = -F(e_n,ab) = 0$
for all $a\in W$ and $b\in R$, and thus $ae_n=\alpha e_n$
for some $\alpha\in\KK$. But $a$ is nilpotent, and $\alpha=0$.
Finally, we have $e_ne_{n+1} = e_ne_1^2 = 0$.
	
Further,
	\begin {equation}
		\label {th2eq1}
		F(e_{n+1}a,1)=-F(e_{n+1},a)=-F(e_1^2,a)=-F(1,e_1^2a)=-F(1,e_{n+1}a)
\Rightarrow F(e_{n+1},a) = 0
	\end {equation}
for all	$a\in W$.
	\begin {enumerate}
	\item Let $n\ge 3$. We claim that $e_{n+1}\mm=0$.
	  Indeed, for $1\le i\ne j\le n-1$ we have
		$$
			e_{n+1}e_i = (e_j^2 - \lambda_{jj}e_n)e_i = \lambda_{ij}e_je_n -\lambda_{jj}e_ne_i = 0.
		$$
In this case the algebra $R$ is isomorphic to
$$
\KK[e_1,\ldots,e_n]/(e_ie_j - \lambda_{ij}e_n, e_i^2-e_j^2-(\lambda_{ii}-\lambda_{jj})e_n,
e_se_n, \, 1\le i<j\le n-1, 1\le s\le n).
$$
	\item Let $n = 2$.
		We have $e_3^2 = e_1^4 \in \mathfrak m^4 = 0 \Rightarrow e_3^2 = 0$.
		Since $F(e_1e_3, 1) = -F(e_3, e_1) = 0$, it follows that $e_1e_3 \in W$.
		Thus $e_1e_3 = \alpha e_1 + \beta e_2$ and
		we have
		$$
			0 = e_1^4 = (e_1e_3)e_1 = \alpha e_1^2 + \beta e_1e_2 = \alpha e_3 \Rightarrow \alpha =0.
		$$
If $\beta = 0$, then $R\cong\KK[e_1,e_2]/(e_1^3, e_1e_2,e_2^2)$.		
If $\beta \ne 0$, then we may assume that $\beta = 1$, and
$R \simeq \KK[e_1]/(e_1^4)$ with $e_2=e_1^3, e_3=e_1^2$.
	\end{enumerate}
	
In all cases $e_{n+1}^2 = e_1^2e_{n+1} = 0$, and it follows that
$F(e_{n+1}, e_{n+1}) = F(1, e_{n+1}^2) = 0$.
Combining this with (\ref{th2eq1}), we obtain
	$$
		M(F) =
\begin{pmatrix}
				0 			&0 			&\ldots	&0	&0 &1\\
				0 			&-1			&\ldots	&0 	&0 &0\\
				\vdots 	&\vdots			&\ddots	&\vdots	&\vdots &\vdots\\
				0 	    &0 	&\ldots &-1 		&0	&0\\
				0       &0  &\ldots         &0  &0 &0\\
        1			  &0	&\ldots			&0	&0 &0\\
\end{pmatrix}.
	$$
Proposition~\ref{prco} is proved.
\end{proof}
	
\begin{remark}
\label {th3rm1}
The normal form of a symmetric matrix $\Lambda$ is unique up to permutation of blocks.
Indeed, we conjugate the matrix $\Lambda$ by the symmetric block diagonal matrix $T$ such that each block $T_k$ is
		$$
			\frac12
			\begin{pmatrix}
			 	1 & 0 & \dots & 0 & i\\
				 0 & 1 & & i & 0\\
				 \vdots & & \ddots & & \vdots\\
				 0 & i & & 1 & 0\\
				 i & 0 & \dots & 0 &1\\
			\end{pmatrix},
		$$
and obtain the Jordan normal form of $\Lambda$ with the same block sizes and the same eigenvalues.

We claim that the matrix $\Lambda$ defining a triple $(R,W,F)$ is unique up to
permutation of blocks, scalar multiplication, and adding a scalar matrix.
To see this, let two matrices $\Lambda, \Lambda'$  define the same triple $(R,W,F)$.
Notice that adding a scalar matrix to $\Lambda$ we do not change the defining relations of $R$.
Denote by $\phi$ an automorphism of $R$ such that $W = \phi (W)$ and
$$
			F = \phi^{-1T}F\phi^{-1}.
$$
It yields $\Ker F = \phi(\Ker F)$ and
$\phi(e_n) = \alpha e_n$. Multiplying the matrix $\Lambda'$ by $\alpha^{-1}$
we obtain $\phi (e_n) = e_n$. Moreover, $\phi$ induces on $W/\Ker F$ an
orthogonal transformation, and thus two canonical forms of the matrix $\Lambda$
can differ only by the order of blocks.
\end{remark}

\begin{example}
Two cases in Proposition~\ref{prco}, 2, correspond to two non-equivalent actions of
$\GG_a^2$ on the quadric $2x_0x_3-x_1^2=0$ in $\PP^3$, namely,
$$
(a_1,a_2)\cdot[x_0:x_1:x_2:x_3]=[x_0:x_1+a_1x_0:x_2+a_2x_0:x_3+\frac{a_1^2}{2}x_0+a_1x_1]
$$
and
$$
(a_1,a_2)\cdot[x_0:x_1:x_2:x_3]=[x_0:x_1+a_1x_0:
x_2+\left(a_2+\frac{a_1^3}{6}\right)x_0+\frac{a_1^2}{2}x_1+a_1x_3:
x_3+\frac{a_1^2}{2}x_0+a_1x_1].
$$
For the first action there is a line of fixed points, while the second
one has three orbits.
\end{example}

\begin{example}
Let $n=3$. If the matrix $\Lambda$ is diagonal, then up to scalar addition and
multiplication we have  $\Lambda =  \begin{pmatrix} 0 & 0 \\	0 & 0\\	\end{pmatrix}$
or $\Lambda =  \begin{pmatrix} 0 & 0 \\	0 & 1\\	\end{pmatrix}$. With non-diagonal
$\Lambda$ we have $\begin{pmatrix}i/2 & 1/2\\1/2 & -i/2\\ \end{pmatrix}$.
So there are three equivalence classes of additive actions in this case,
and they can be easily written down explicitly.
\end{example}

\begin{example} Consider the case $n=4$. We have six types of the matrix $\Lambda$ with one
depending on a parameter. Namely, in the diagonal matrix
$\Lambda = \text{diag} (0, 1, t)$, where $t \in \KK\setminus \{0, 1\}$, the parameter $t$
is defined up to transformations
$
\{t, \frac{1}{t}, 1-t, \frac{t-1}{t}, \frac{t}{t-1}, \frac{1}{1-t}\}
$.
Therefore, the parameters $t$ and $t'$ determine equivalent actions if and only if
$$
\frac{(t^2 - t + 1)^3}{t^2(1-t)^2} = \frac {(t'^2 - t' + 1)^3}{t'^2(1-t')^2}.
$$
The action of $\GG_a^4$ on the quadric $2x_0x_5-x_1^2-x_2^2-x_3^2=0$ in this case has the form
\begin{multline*}
(a_1,a_2,a_3,a_4)\cdot[x_0:x_1:x_2:x_3:x_4:x_5]=
[x_0 : x_1+a_1x_0 : x_2 + a_2x_0 : x_3 + a_3x_0 : \\
: x_4 + \frac{2a_4 + a_2^2 + ta_3^2}{2}x_0 + a_2x_2 + ta_3x_3:
x_5 + \frac{a_1^2 + a_2^2 + a_3^2}{2}x_0 + a_1x_1 + a_2x_2 + a_3x_3].
\end{multline*}
This agrees with the results of \cite[Section~4]{ASh}.
\end{example}


\section*{Acknowledgement}

The authors are grateful to Ernest B.~Vinberg for useful consultations.
Special thanks are due to the referees and the editors for careful reading,
constructive criticism and important suggestions.


%
\end{document}